\documentclass[11pt, a4paper, draft]{amsart}

\usepackage{amsfonts, amsthm, amsmath, amscd, amssymb}
\usepackage[draft=false]{hyperref}
\usepackage[T1]{fontenc}
\usepackage{ae}
\usepackage[all]{xy}

\newcommand\QQ{\mathbb{Q}}
\newcommand\CC{\mathbb{C}}
\newcommand\RR{\mathbb{R}}
\newcommand\NN{\mathcal{N}}
\newcommand\N{\mathbf{N}}
\newcommand\TT{\mathcal{T}}
\newcommand\ZZ{\mathbb{Z}}
\newcommand\ZZp{\ZZ_{>0}}
\newcommand\ZZnz{\ZZ_{\ne 0}}
\newcommand\PP{\mathbb{P}}
\newcommand\xx{\mathbf{x}}
\newcommand\kk{\mathbf{k}}
\newcommand\EE{\mathcal{E}}
\newcommand\dd{\,\mathrm{d}}
\newcommand{\base}[4]{\ee^{({#1},{#2},{#3},{#4})}}
\newcommand{\Base}[4]{\N^{({#1},{#2},{#3},{#4})}}
\newcommand{\congr}[3]{{#1} \equiv {#2}\ (\mathrm{mod}\ {#3})}
\newcommand{\cp}[2]{{\gcd(#1,#2)=1}}
\newcommand{\Atwo}{{\mathbf A}_2}
\newcommand{\tS}{{\widetilde S}}
\renewcommand{\le}{\leqslant}
\renewcommand{\ge}{\geqslant}
\newcommand{\ee}{\boldsymbol{\eta}}
\renewcommand{\aa}{\boldsymbol{\alpha}}
\newcommand{\ex}[1]{*+<10pt>[o][F]{#1}}
\newcommand\e{\eta}
\newcommand\al{\alpha}
\newcommand\ep{\varepsilon}
\newcommand\RE{\Re e}
\newcommand\phis{\phi^*}
\newcommand\phid{\phi^\dagger}
\newcommand\Ga{\mathbb{G}_\mathrm{a}}
\renewcommand{\theta}{\vartheta}

\newtheorem*{theorem}{Theorem}
\newtheorem{lemma}{Lemma}
\numberwithin{equation}{section}

\begin{document}

\title[Manin's conjecture for a quintic del Pezzo surface] {Manin's
  conjecture for a quintic del Pezzo surface with $\Atwo$ singularity}

\author{Ulrich Derenthal} 

\address{Institut f\"ur Mathematik, Universit\"at
  Z\"urich, Winterthurerstrasse 190, 8057 Z\"urich, Switzerland}

\email{ulrich.derenthal@math.unizh.ch}

\begin{abstract}
  Manin's conjecture is proved for a split del Pezzo surface of degree 5 with
  a singularity of type $\Atwo$.
\end{abstract}


\maketitle

\tableofcontents

\section{Introduction}\label{sec:intro}

Let $S \subset \PP^5$ be the del Pezzo surface of degree $5$ defined by
\begin{equation}\label{eq:surface}
  \begin{split}
    &x_0x_2-x_1x_5 = x_0x_2-x_3x_4 = x_0x_3+x_1^2+x_1x_4\\ =
    {}&x_0x_5+x_1x_4+x_4^2 = x_3x_5+x_1x_2+x_2x_4 = 0.
  \end{split}
\end{equation}
It contains a unique singularity of type $\Atwo$ and four lines, all
of them defined over $\QQ$. Let $U \subset S$ be the complement
of these lines.

We define the height of any rational point $\xx \in S(\QQ)$ that is
represented by integral and relatively coprime coordinates $(x_0,
\dots, x_5)$ as
\[H(\xx) := \max\{|x_0|,\dots,|x_5|\}.\] For any $B\ge 1$, let
\[N_{U,H}(B) := \#\{\xx \in U(\QQ) \mid H(\xx) \le B\}\] be the number
of rational points in $U$ whose height is at most $B$.

We prove the following result:

\begin{theorem}
  We have
  \[N_{U,H}(B) = c_{S,H}B(\log B)^4+O(B(\log B)^{4-1/5}),\] where
  \[c_{S,H}=\frac{1}{864}\cdot \omega_\infty\cdot\prod_p\left(1-\frac 1
    p\right)^5 \left(1+\frac 5 p+\frac 1 {p^2}\right)\] and
  \[\omega_\infty=\int_{|t_5|,|t_1|,|t_1t_5^2t_6^2+t_1^2t_6|,|t_1t_5t_6|,|t_5^2t_6+t_1|,|t_5^3t_6^2+t_1t_5t_6|\le 1,\ t_5 > 0} \dd t_1 \dd t_5 \dd t_6.\]
\end{theorem}

Manin's conjecture \cite{MR89m:11060} predicts that $N_{U,H}(B)$ grows as $cB(\log B)^{k-1}$ for
$B \to \infty$ where $k$ is the rank of the Picard group of the minimal
desingularization $\tS$ of $S$. As $S$ is a del Pezzo surface of degree 5 whose
lines are defined over $\QQ$, we have $k=5$, so our result agrees with this
conjecture.

Peyre \cite{MR1340296} predicts that $c$ is a product of a constant
$\alpha(\tS)$ whose value is $1/864$ by \cite{MR2318651} and
\cite{math.AG/0703202} and a product of local densities. We expect
that $(1-1/p)^5(1+5/p+1/p^2)$ agrees with the density at each prime
$p$, and that $\omega_\infty$ agrees with the real density, but we do
not check this here.

Note that $S$ is neither toric nor an equivariant compactification of
$\Ga^2$, so our theorem is not a consequence of \cite{MR1620682} or
\cite{MR1906155}.

For the proof of the theorem, we use the basic strategy of \cite{MR2320172},
\cite{MR2332351} and \cite{MR2290499} together with the techniques
introduced in \cite{arXiv:0710.1560}. In Section~\ref{sec:torsor}, we translate the
counting problem to the question of integral points on a universal torsor and
split their counting into three parts. As outlined at the end of
Section~\ref{sec:torsor}, these parts are handled separately in
Sections~\ref{sec:Nb2} to \ref{sec:Nb1} and put together again in
Section~\ref{sec:final} to complete the proof of the theorem.

\section{A universal torsor}\label{sec:torsor}

We use the notation \[\ee=(\e_1,\dots,\e_4), \quad \ee' = (\e_1,
\dots, \e_6), \quad \aa = (\al_1, \al_2)\] and, for $(n_1, \dots, n_4)
\in \QQ^4$, \[\base{n_1}{n_2}{n_3}{n_4} :=
\e_1^{n_1}\e_2^{n_2}\e_3^{n_3}\e_4^{n_4}.\]

By the method of \cite{MR2290499} and using the data of \cite{math.AG/0604194}
on the geometry of $S$ and its minimal desingularization $\tS$, we obtain a
bijection $\Psi : \TT \to U(\QQ)$ with
\[\TT := \{(\ee',\aa) \in \ZZ^5 \times \ZZp \times \ZZ^2 \mid
\text{\eqref{eq:torsor} and coprimality conditions hold} \}\] where
\begin{equation}\label{eq:torsor} \e_4\e_5^2\e_6 + \e_1\al_1 + \e_2\al_2=0
\end{equation}
and the coprimality conditions are described by the extended Dynkin
diagram of $E_1, \dots, E_6, A_1, A_2$ in Figure~\ref{fig:dynkin},
using the rule that two variables are coprime unless the corresponding
divisors in the diagram are connected by an edge.  The map $\Psi$
sends $(\ee',\aa) \in \TT$ to
\[(\base 2 2 3 2 \e_5, \base 2 1 2 1 \al_1, \e_6\al_1\al_2,
\base 1 0 1 1 \e_5\e_6\al_1, \base 1 2 2 1 \al_2, \base 0 1 1 1 \e_5\e_6\al_2)\]
in $U(\QQ)$.

\begin{figure}[ht]
  \centering
  \[\xymatrix{ A_1 \ar@{-}[rrr]\ar@{-}[dr]\ar@{-}[dd]&&& E_1\ar@{-}[dr]\\
      & E_6 \ar@{-}[r]
      & E_5\ar@{-}[r] & \ex{E_4}\ar@{-}[r] & \ex{E_3}\\
      A_2\ar@{-}[rrr]\ar@{-}[ur] &&& E_2\ar@{-}[ur]}\]
  \caption{Configuration of curves on $\tS$.}
  \label{fig:dynkin}
\end{figure}
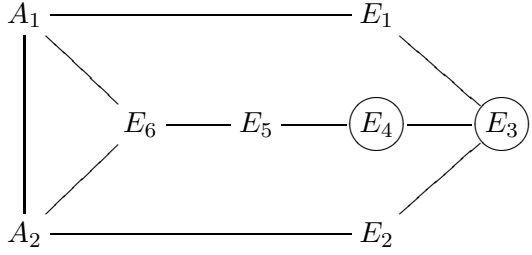

Note that these coprimality conditions imply that the formula above
for $\Psi(\ee',\al)$ results in relatively coprime coordinates
$\Psi(\ee',\al)_i$, so \[H(\Psi(\ee',\al)) =
\max_i\{|\Psi(\ee',\al)_i|\}.\]
With \eqref{eq:torsor}, $H(\Psi(\ee',\aa)) \le B$ implies
\begin{equation}
  \label{eq:further_height}
  \base 1 1 2 2 \e_5^2|\e_6| \le 2B, \quad \base 0 0 1 2 \e_5^3 |\e_6|^2 \le 2B.
\end{equation}

Using \eqref{eq:torsor}, the coprimality conditions can be rewritten as
\begin{align}
  \label{eq:cpal2} &\cp{\al_2}{\e_3\e_5},\\
  \label{eq:cpal1} &\cp{\al_1}{\e_3\e_4},\\
  \label{eq:cpe6} &\cp{\e_6}{\e_1\e_2\e_3\e_4},\\
  \label{eq:cpe5} &\cp{\e_5}{\e_1\e_2\e_3},\\
  \label{eq:cpe} &\cp{\e_1}{\e_2},\ \cp{\e_1}{\e_4},\ \cp{\e_2}{\e_4}.
\end{align}

Therefore, the number $N_{U,H}(B)$ coincides with the number of
$(\ee',\aa) \in \ZZp^5 \times \ZZnz \times \ZZ^2$ which satisfy the
torsor equation \eqref{eq:torsor}, the coprimality conditions
\eqref{eq:cpal2}--\eqref{eq:cpe} and the height condition
$H(\Psi(\ee',\aa)) \le B$.

Our further strategy is as follows. For fixed $\ee'$, we estimate the
number of $\aa$ satisfying the torsor equation, the coprimality
conditions and the height condition. We sum this number over all
suitable $\ee'$ afterwards. To get a hold of the error terms in these
summations, it will be useful to do this summations in different
orders depending on the relative size of $\e_1, \dots, \e_6$.

We denote the number of $(\ee',\aa)$ contributing to $N_{U,H}(B)$ that fulfill
\begin{equation}
  \label{eq:bige5}
  |\e_5| \ge |\e_6|
\end{equation}
by $N_a(B)$, and the number of those satisfying
\begin{equation}
  \label{eq:bige6}
  |\e_5| < |\e_6|.
\end{equation}
by $N_b(B)$.

We split the elements contributing to $N_b(B)$ further into two subsets: For
some $A>0$ to be chosen in Section~\ref{sec:final}, let $N_{b_1}(B;A)$ be the
number of $(\ee',\aa)$ satisfying \eqref{eq:bige6} and
\begin{equation}\label{eq:logBA}
  \base 2 2 3 2 \le \frac{B}{(\log B)^A},
\end{equation}
while $N_{b_2}(B;A)$ is the number of the remaining ones.

We deal with $N_{b_2}(B;A)$ in the following Section~\ref{sec:Nb2}. As
a first step for both $N_a(B)$ and $N_{b_1}(B;A)$, we estimate the
number of $\aa$ in Section~\ref{sec:N0}. For $N_a$, we sum first over
the bigger $\e_5$ and then over $\e_6$ in Section~\ref{sec:Na}, while
for $N_{b_1}(B;A)$, we sum in the reverse order in
Section~\ref{sec:Nb1}. The resulting main terms are put together and
summed over the remaining variables $\e_1, \dots, \e_4$ in
Section~\ref{sec:final} to complete the proof of the theorem.

\section{Estimating $N_{b_2}(B;A)$}\label{sec:Nb2}

Our strategy is to estimate the number of $(\ee',\al)$ lying in dyadic
intervals first, and to sum over all possible intervals in a second
step.

\begin{lemma}\label{lem:Nb2}
  We have $N_{b_2}(B;A) \ll_A B(\log B)^3(\log \log B)^2$.
\end{lemma}

\begin{proof}
  Let $\NN=\NN(N_1,\dots,N_6,A_1,A_2)$ be the number of $(\ee',\aa)$
  subject to $N_i/2 < |\e_i| \le N_i$ for $i \in \{1, \dots, 6\}$ and
  $A_j/2 < |\al_j| \le A_j$ for $j \in \{1,2\}$.

  Because of the height conditions and using the notation $\Base
  {n_1}{n_2}{n_3}{n_4} := N_1^{n_1}N_2^{n_2}N_3^{n_3}N_4^{n_4}$,
  we have, if $\NN > 0$,
  \begin{align}
    \label{eq:heightN1}&B(\log B)^{-A} \ll \Base 2 2 3 2 \ll B,\\
    \label{eq:heightN2}&N_6A_1A_2 \ll B,\\
    \label{eq:heightN3}&\Base 1 0 1 1 N_5N_6A_1 \ll B,\\
    \label{eq:heightN4}&\Base 0 1 1 1 N_5N_6A_2 \ll B,\\
    \label{eq:heightN5}&\Base 1 1 2 2 N_5^2N_6 \ll B,\\
    \label{eq:heightN6}&N_5 \ll (\log B)^A.
  \end{align}
  Here, \eqref{eq:heightN5} follows from \eqref{eq:further_height}.
  As in \cite[Lemma~5,~6]{arXiv:0710.1560}, we obtain by estimating the number of $\al_1,
  \al_2$ in two ways first and summing over $\e_1,\dots,\e_6$ afterwards:
  \[\NN \ll N_3N_4N_5N_6(N_1A_1)^{1/2}(N_2A_2)^{1/2} +
  N_1N_2N_3N_4N_5N_6.\]

  Next, we sum this estimate for $\NN(N_1, \dots, N_6, A_1, A_2)$ over all
  possible dyadic intervals, with $N_1, \dots, N_6, A_1, A_2$ subject to
  \eqref{eq:heightN1}--\eqref{eq:heightN6}.

  For the first term, we have using \eqref{eq:heightN2}--\eqref{eq:heightN5}
  \[\begin{split}
    &\sum_{N_1, \dots, N_6,A_1,A_2} \Base {1/2} {1/2} 1 1 N_5 N_6 A_1^{1/2}
    A_2^{1/2}\\
    &\ll B^{1/4} \sum_{N_1, \dots, N_6,A_1,A_2} \Base {1/2} {1/2} 1 1 N_5
    N_6^{3/4} A_1^{1/4} A_2^{1/4}\\
    &\ll B^{3/4} \sum_{N_1,\dots, N_6} \Base
    {1/4}{1/4}{1/2}{1/2}N_5^{1/2}N_6^{1/4}\\
    &\ll B\sum_{N_1,\dots,N_5} 1\\
    &\ll_A B(\log B)^3(\log \log B)^2.
  \end{split}\]
Here we have used that for fixed $N_2,N_3,N_4$, there are
only $O_A(\log \log B)$ possibilities for $N_1$ and $N_5$ by
\eqref{eq:heightN1} and \eqref{eq:heightN6}.

For the second term, we use \eqref{eq:heightN5} to obtain 
\[\begin{split}
  \sum_{N_1,\dots,N_6,A_1,A_2} \Base 1 1 1 1 N_5 N_6
  &\ll B \sum_{N_1,\dots,N_5,A_1,A_2} \frac{1}{N_3N_4N_5}\\ 
  &\ll_A B (\log B)^3 (\log \log B),
\end{split}\]
which completes the proof.
\end{proof}

\section{Real-valued functions}

Let
\begin{equation}\label{eq:h}
h(t_0,t_1,t_5,t_6) :=
\max\left\{
  \begin{aligned}
    &|t_0^4t_5|,|t_0^4t_1|,|t_1t_5^2t_6^2+t_0^2t_1^2t_6|,|t_0^2t_1t_5t_6|,\\
    &|t_0^2t_5^2t_6+t_0^4t_1|,|t_5^3t_6^2+t_0^2t_1t_5t_6|
  \end{aligned}
\right\}.
\end{equation}
Defining
\begin{align*}
  Y_0&:=\left(\frac{\base 2 2 3
    2}{B}\right)^{1/5},& Y_1&:=\left(\frac{B}{\base
    2{-3}{-2}{-3}}\right)^{1/5},\\Y_5&:=Y_0^{-1}, &
Y_6&:=\left(\frac{B}{\base {-3}{-3}{-2}{2}}\right)^{1/5},
\end{align*} we note that the height condition $H(\Psi(\ee',\al)) \le
B$ is equivalent to \[h(Y_0,\al_1/Y_1,\e_5/Y_5,\e_6/Y_6) \le 1.\]

Define
\begin{align}
  \label{eq:g0} g_0(t_0,t_5,t_6) &:= \int_{h(t_0,t_1,t_5,t_6) \le 1} 1 \dd t_1,\\
  \label{eq:g1a} g_1^a(t_0,t_6;\ee;B) &:= \int_{Y_5t_5\ge |Y_6t_6|, t_5>0}
  g_0(t_0,t_5,t_6) \dd t_5,\\
  \label{eq:g1b} g_1^b(t_0,t_5;\ee;B) &:= \int_{|Y_6t_6|>\max\{Y_5t_5,1\}}
  g_0(t_0,t_5,t_6) \dd t_6,\\
  \label{eq:g2a} g_2^a(t_0;\ee;B) &:= \int_{|Y_6t_6|>1} g_1^a(t_0,t_6;\ee;B)
  \dd t_6,\\
  \label{eq:g2b} g_2^b(t_0;\ee;B) &:= \int_0^\infty g_1^b(t_0,t_5;\ee;B) \dd
  t_5.
\end{align}
We have
\begin{equation}\label{eq:g2}
\begin{split}
g_2(t_0;\ee;B) &:=g_2^a(t_0;\ee;B)+ g_2^b(t_0;\ee;B)\\ &=
\int_{h(t_0,t_1,t_5,t_6) \le 1, |Y_6t_6|>1, t_5>0} \dd t_1 \dd t_5 \dd t_6.
\end{split}\end{equation}

\begin{lemma}\label{lem:g}
  Let $\ee \in \ZZp^4$ be given. Then we have:
  \begin{enumerate}
  \item\label{it:g0} $g_0(t_0,t_5,t_6) \ll \frac{1}{t_0|t_6|^{1/2}}$.
  \item\label{it:g1a} $g_1^a(t_0,t_6;\ee;B) \ll \int_0^\infty g_0(t_0,t_5,t_6)
    \dd t_5 \ll \min\{\frac{1}{t_0^{1/2}|t_6|^{5/4}},\frac{1}{t_0^8}\}$.
  \item\label{it:g1b} $g_1^b(t_0,t_5;\ee;B) \ll \int_{-\infty}^\infty
    g_0(t_0,t_5,t_6) \dd t_6 \ll \frac{1}{t_0t_5^{3/4}}$.
  \end{enumerate}
\end{lemma}

\begin{proof}
  Since $h(t_0,t_1,t_5,t_6) \le 1$ implies $t_1 \le t_0^{-4}$ and $t_5 \le
  t_0^{-4}$, the second bound of \eqref{it:g1a} holds. 

  It is not hard to check that given $a, b \in \RR \setminus \{0\}$, the
  condition $|at_1^2+bt_1| \le 1$ describes a set of $t_1$ whose length is $\ll
  |a|^{-1/2}$ for $b^2 \le 8|a|$, while its length is $\ll |b|^{-1}
  \ll |a|^{-1/2}$ for $b^2 > 8|a|$.
  
  We apply this for $a=t_0^2t_6$ and $b=t_5^2t_6^2$, which gives
  $g_0(t_0,t_5,t_6) \ll (t_0^2|t_6|)^{-1/2}$ which is \eqref{it:g0}.
  Integrating it over $t_6 \ll t_5^{3/2}$ (which holds since $|t_0^2t_1t_5t_6|
  \le 1$ and $|t_5^3t_6^2+t_0^2t_1t_5t_6| \le 1$ imply $|t_5^3t_6^2| \le 2$)
  results in \eqref{it:g1b}.

  For the first bound of \eqref{it:g1a}, we distinguish the case $t_5^4t_6^4
  \le 8t_0^2|t_6|$ and its opposite. In the first case, we combine $t_5 \ll
  t_0^{1/2}|t_6|^{-3/4}$ with \eqref{it:g0}. In the second case, we integrate
  $g_0(t_0,t_5,t_6) \ll t_5^{-2}|t_6|^{-2}$ over $t_5 \gg
  t_0^{1/2}|t_6|^{-3/4}$.
  \end{proof}

Finally, we define
\begin{equation}\label{eq:G}
  G_2(t_0):=\int_{h(t_0,t_1,t_5,t_6)\le 1, t_5 > 0} \dd t_1 \dd t_5 \dd t_6
\end{equation}
which is related to $\omega_\infty$ defined in the statement of our
theorem:

\begin{lemma}\label{lem:G}
  For any $t_0>0$, we have $G_2(t_0) = \frac{\omega_\infty}{t_0^2}$.
\end{lemma}

\begin{proof}
  Similar to \cite[Lemma~7]{arXiv:0710.1560}.
\end{proof}

\section{Estimating $N_a(B)$ and $N_{b_1}(B;A)$ -- first step}\label{sec:N0}

For fixed $\ee'$ subject to the coprimality conditions
\eqref{eq:cpe6}--\eqref{eq:cpe}, let $N_0$ be the number of $\al_1, \al_2$
subject to \eqref{eq:torsor}, $h(Y_0,\al_1/Y_1,\e_5/Y_5,\e_6/Y_6) \le
1$ and the coprimality conditions \eqref{eq:cpal2}, \eqref{eq:cpal1}.

We remove \eqref{eq:cpal2} by a M\"obius inversion and obtain
\[N_0 = 
\sum_{k_2|\e_3\e_5} \mu(k_2) \#\left\{\al_1\ \Bigg|\ 
  \begin{aligned}
    &\congr{\e_4\e_5^2\e_6}{-\e_1\al_1}{k_2\e_2},\\
    &h(Y_0,\al_1/Y_1,\e_5/Y_5,\e_6/Y_6) \le 1,\\
    &\text{\eqref{eq:cpal1} holds}
  \end{aligned}
\right\}.\] The summand vanishes unless $\cp{k_2}{\e_1\e_4}$. Since
$\e_3,\e_5$ are coprime, we write $k_2=k_{23}k_{25}$ uniquely such that
$k_{2i} \mid \e_i$ for $i \in \{3,5\}$. We check that $k_{25}|\al_1$. We
write $\e_5=k_{25}\e_5'$, $\al_1=k_{25}\al_1'$ and obtain
\[N_0=\sum_{\substack{k_{23}|\e_3, k_{25}|\e_5\\\cp{k_{23}k_{25}}{\e_1\e_4}}}
\mu(k_{23})\mu(k_{25}) N_0(k_{23},k_{25})\] where \[N_0(k_{23},k_{25}) =
\#\left\{\al_1' \ \Bigg|\  \begin{aligned}
    &\congr{k_{25}\e_4\e_5'^2\e_6}{-\e_1\al_1'}{k_{23}\e_2},\\
    &h(Y_0,\al_1'k_{25}/Y_1,\e_5/Y_5,\e_6/Y_6) \le 1,\\
    &\cp{k_{25}\al_1'}{\e_3\e_4}
  \end{aligned}\right\}.\] 
Note that $\cp{k_{25}}{\e_3\e_4}$ holds automatically, so we may remove this
condition. We remove the coprimality condition for $\al_1'$ by another
M\"obius inversion and obtain, writing
$\al_1'=k_1\al_1''$,
\[N_0(k_{23},k_{25})=\sum_{k_1|\e_3\e_4} \mu(k_1) \#\left\{\al_1''
  \ \Big|\  \begin{aligned}
    &\congr{k_{25}\e_4\e_5'^2\e_6}{-k_1\e_1\al_1''}{k_{23}\e_2},\\
    &h(Y_0,\al_1''k_{25}k_1/Y_1,\e_5/Y_5,\e_6/Y_6) \le 1,
  \end{aligned}\right\}.\] 
Note that the summand vanishes unless $\cp{k_1}{k_{23}\e_2}$, so we may
restrict the summation over $k_1|\e_3\e_4$ subject to $\cp{k_1}{k_{23}\e_2}$.
Since then $\cp{k_1\e_1}{k_{23}\e_2}$, the number of $\al_1''$ is
\[\frac{Y_1}{k_1k_{23}k_{25}\e_2} g_0(Y_0,\e_5/Y_5,\e_6/Y_6) +
O(1).\]

Define $\phis(n):=\prod_{p|n}(1-1/p)$.

\begin{lemma}\label{lem:sum0}
  We have
  \[N_0 = \frac{Y_1}{\e_2} g_0(Y_0,\e_5/Y_5,\e_6/Y_6)\theta_0(\ee)
  \frac{\phis(\e_5)}{\phis(\gcd(\e_5,\e_4))} +O(R_0(\ee, \e_5, \e_6))\] with
  \[\theta_0(\ee) := \sum_{\substack{k_{23}|\e_3\\\cp{k_{23}}{\e_1\e_4}}}
  \frac{\mu(k_{23})\phis(\e_3\e_4)} {k_{23}\phis(\gcd(\e_3,k_{23}\e_2))}\] and
  \[\sum_{\e_1, \dots, \e_6} R_0(\ee,\e_5,\e_6) \ll B(\log
  B)^2.\]
\end{lemma}

\begin{proof}
  For the main term, note that
  \[\begin{split}
    &\sum_{\substack{k_{23}|\e_3, k_{25}|\e_5\\
        \cp{k_{23}k_{25}}{\e_1\e_4}}}
    \frac{\mu(k_{23})\mu(k_{25})}{k_{23}k_{25}}
    \sum_{\substack{k_1|\e_3\e_4\\\cp{k_1}{k_{23}\e_2}}}
    \frac{\mu(k_1)}{k_1}\\
    ={}&\sum_{\substack{k_{23}|\e_3\\
        \cp{k_{23}}{\e_1\e_4}}} \frac{\mu(k_{23})}{k_{23}}\cdot
    \frac{\phis(\e_5)}{\phis(\gcd(\e_5,\e_1\e_4))}\cdot
    \frac{\phis(\e_3\e_4)}{\phis(\gcd(\e_3\e_4,k_{23}\e_2))}.
  \end{split}\] 
  Using $\cp{\e_5}{\e_1}$ and $\cp{\e_4}{k_{23}\e_2}$, we
  obtain $\theta_0$.

  We have \[R_0(\ee,\e_5,\e_6) \ll
  2^{\omega(\e_3)+\omega(\e_5)+\omega(\e_3\e_4)}.\] We sum this over all
  suitable $\e_1, \dots, \e_6$ and use \eqref{eq:further_height} to obtain
  \[\begin{split}
    \sum_{\e_1, \dots, \e_6} R_0(\ee,\e_5,\e_6) &\ll \sum_{\e_1, \dots,
      \e_5}
    \frac{2^{\omega(\e_3)+\omega(\e_5)+\omega(\e_3\e_4)}B}{\base
      1 1 2 2 \e_5^2} \\
    & \ll B(\log B)^2,
  \end{split}\]
  completing the proof of this lemma.
\end{proof}

\section{Estimating $N_a(B)$ -- second step}\label{sec:Na}

Let $N_1^a:=N_1^a(\ee, \e_6;B)$ be the sum of the main term of
Lemma~\ref{lem:sum0} over $\e_5$ subject to \eqref{eq:cpe5} and
\eqref{eq:bige5}. We sum the main term of $N_1^a$ over $\e_6$ afterwards to
obtain $N_2^a(\ee;B)$.

Using \cite[Lemma~2]{arXiv:0710.1560} with $\alpha=0,\,q=1$, we obtain
(where $f_{a,b}(n)$ is defined to be $\phis(n)/\phis(\gcd(n,a))$ if
$\cp{n}{b}$ and to be zero otherwise)
\[
\begin{split}
  N_1^a={}&\frac{Y_1}{\e_2}\theta_0(\ee) \sum_{\e_5\ge|\e_6|}
  f_{\e_4,\e_1\e_2\e_3}(\e_5) g_0(Y_0,\e_5/Y_5,\e_6/Y_6;\ee;B)\\
  ={}&\frac{Y_1Y_5}{\e_2} g_1^a(Y_0,\e_6/Y_6;\ee;B) \theta_0(\ee)
  \frac{\phis(\e_1\e_2\e_3)}{\zeta(2)}\prod_{p|\e_1\e_2\e_3\e_4}\left(1-\frac 1 {p^2}\right)^{-1}\\
  &+O\left(\frac{Y_1}{\e_2}|\theta_0(\ee)|(\log
    B)2^{\omega(\e_1\e_2\e_3)} \sup_{t_5}
    g_0(Y_0,t_5,\e_6/Y_6;\ee;B)\right),
\end{split}\] where the supremum is taken over $t_5\ge|\e_6|/Y_5$.

\begin{lemma}\label{lem:sum1a}
  We have \[N_1^a=\frac{Y_1Y_5}{\e_2} g_1^a(Y_0,\e_6/Y_6;\ee;B)
  \theta_1^a(\ee) + O(R_1^a(\ee,\e_6;B))\]
  with
  \[\theta_1^a(\ee):=\theta_0(\ee)\frac{\phis(\e_1\e_2\e_3)}{\zeta(2)} 
  \prod_{p|\e_1\e_2\e_3\e_4}\left(1-\frac 1 {p^2}\right)^{-1}\]
  and
  \[\sum_{\ee,\e_6}R_1^a(\ee,\e_6;B) \ll B \log B.\]
\end{lemma}

\begin{proof}
  The main term is clear. Define $\phid(n):=\prod{p|n}(1+1/p)$. For the error term, we use
  Lemma~\ref{lem:g}\eqref{it:g0} to estimate its sum over $\ee, \e_6$ as
  \[\begin{split}
    &\ll \sum_{\ee,\e_6}
    \frac{2^{\omega(\e_1\e_2\e_3)}\phid(\e_3)Y_1Y_6^{1/2}\log B}
    {Y_0\e_2|\e_6|^{1/2}}\\
    &=
    \sum_{\ee,\e_6}\frac{2^{\omega(\e_1\e_2\e_3)}\phid(\e_3)B^{1/2}\log B}
    {\e_1^{1/2}\e_2^{1/2}|\e_6|^{1/2}}\\
    &\ll \sum_{\e_1,\e_2,\e_3,\e_6}
    \frac{2^{\omega(\e_1\e_2\e_3)}\phid(\e_3)B\log B} {\base
      {5/4}{5/4}{5/4}0|\e_6|^{3/2}}\\
    &\ll B\log B.
  \end{split}\] Here, we use \[\e_4 \le \left(\frac{B}{\base 2 2 3
      0|\e_6|}\right)^{1/4} \cdot \left(\frac{B}{\base 1 1 2 0
      |\e_6|^3}\right)^{1/4} = \frac{B^{1/2}}{\base
        {3/4}{3/4}{5/4}0|\e_6|}\]
      which is obtained with \eqref{eq:bige5}.
\end{proof}

To sum the main term of $N_1^a$ over $\e_6$, we remove the coprimality
condition \eqref{eq:cpe6} by a M\"obius inversion and obtain, writing
$\e_6=k_6\e_6'$ and applying partial summation,
\[\begin{split}
  N_2^a={}&\frac{Y_1Y_5}{\e_2}\theta_1^a(\ee)\sum_{k_6|\e_1\e_2\e_3\e_4}
  \mu(k_6) \sum_{|\e_6'|\ge 1} g_1^a(Y_0,\e_6'k_6/Y_6;\ee;B)\\
  ={}&\frac{Y_1Y_5Y_6}{\e_2}\theta_1^a(\ee)\sum_{k_6|\e_1\e_2\e_3\e_4}
  \frac{\mu(k_6)}{k_6} \int_{|t_6|\ge k_6/Y_6} g_1^a(Y_0,t_6;\ee;B)
  \dd t_6\\
  &+O\left(\frac{Y_1Y_5}{\e_2}|\theta_1^a(\ee)|\sum_{k_6|\e_1\e_2\e_3\e_4}
    |\mu(k_6)| \sup_{|t_6|\ge k_6/Y_6} g_1^a(Y_0,t_6;\ee;B)\right).
\end{split}\]

\begin{lemma}\label{lem:sum2a}
  We have
  \[N_2^a=\frac{Y_1Y_5Y_6}{\e_2}g_2^a(Y_0;\ee;B)\theta_2^a(\ee)
  + O(R_2^a(\ee;B))\]
  with
  \[\theta_2^a(\ee):=\theta_1^a(\ee)\phis(\e_1\e_2\e_3\e_4)\]
  and
  \[\sum_{\ee}R_2^a(\ee;B) \ll B(\log B)^{4-1/5}.\]
\end{lemma}

\begin{proof}
  In order to replace the integral over $|t_6|\ge k_6/Y_6$ in the estimation
  before the statement of the lemma by $g_2^a(Y_0;\ee;B)$, we must
  add
  \[\frac{Y_1Y_5Y_6}{\e_2}\theta_1^a(\ee)\sum_{k_6|\e_1\e_2\e_3\e_4}
  \frac{\mu(k_6)}{k_6} \int_{1/Y_6 < |t_6| < k_6/Y_6}
  g_1^a(Y_0,t_6;\ee;B) \dd t_6\]
  as a second error term.

  We distinguish the case
  \begin{equation}\label{eq:case1}
    \base 3 3 4 2 < \lambda B
  \end{equation}
  for some $\lambda>0$ to be chosen later, giving a total contribution
  $E_1(\lambda)$ to the error term, and its opposite
  \begin{equation}\label{eq:case2}
    \base 3 3 4 2 \ge \lambda B,
  \end{equation}
  contributing in total $E_2(\lambda)$.
  
  Starting with $E_1(\lambda)$, we use the first bound of
  Lemma~\ref{lem:g}\eqref{it:g1a}. For the first error term, we obtain
  \[\begin{split}
    &\ll \sum_{\ee} \sum_{k_6|\e_1\e_2\e_3\e_4}
    \frac{|\mu(k_6)|\phid(\e_3)Y_1Y_5Y_6^{5/4}}
    {k_6^{5/4}\e_2Y_0^{1/2}}\\
    &\ll \sum_{\ee} 
    \frac{\phid(\e_3)B^{3/4}}
    {\base {1/4}{1/4}0{1/2}}\\
    &\ll \sum_{\e_1,\e_2,\e_3} \frac{\phid(\e_3)\lambda^{1/4}B}{\base 1 1 1 0}\\
    &\ll \lambda^{1/4}B(\log B)^3.
  \end{split}\]
  For the second error term, we use
  \[\int_{1/Y_6}^{k_6/Y_6} g_1^a(Y_0,t_6;\ee;B) \dd t_6 \ll
  \int_{1/Y_6}^{k_6/Y_6} \frac{1}{Y_0^{1/2}|t_6|^{5/4}} \dd t_6\ll
    \frac{Y_6^{1/4}}{Y_0^{1/2}}\]
  and obtain
  \[\begin{split}
    &\ll \sum_{\ee} \sum_{k_6|\e_1\e_2\e_3\e_4}
    \frac{|\mu(k_6)|\phid(\e_3)Y_1Y_5Y_6^{5/4}}
    {k_6\e_2Y_0^{1/2}}\\
    &\ll \sum_{\ee} \frac{\phid(\e_1\e_2\e_3\e_4)\phid(\e_3)B^{3/4}}
    {\base {1/4}{1/4}0{1/2}}\\
    &\ll \sum_{\e_1,\e_2,\e_3}
    \frac{\phid(\e_1\e_2\e_3)\phid(\e_3)\lambda^{1/4}B}
    {\base 1 1 1 0}\\
    &\ll \lambda^{1/4}B(\log B)^3.
  \end{split}\]
  Therefore, $E_1(\lambda) \ll \lambda^{1/4}B(\log B)^3$.
  
  For $E_2(\lambda)$, we use the second bound of
  Lemma~\ref{lem:g}\eqref{it:g1b}. For the first part of this error term, we
  get \[\begin{split} &\ll \sum_{\ee} \sum_{k_6|\e_1\e_2\e_3\e_4}
    \frac{|\mu(k_6)|\phid(\e_3)Y_1Y_5}{\e_2Y_0^8}\\
    &\ll \sum_{\ee} \frac{2^{\omega(\e_1\e_2\e_3\e_4)}\phid(\e_3)
      B^2}{\base 4 4 5 3}\\
    &\ll \sum_{\e_1,\e_2,\e_3} \frac{2^{\omega(\e_1\e_2\e_3)}\phid(\e_3) B\log
      B}{\lambda\base 1 1 1 0}\\
    &\ll \lambda^{-1}B(\log B)^7.
  \end{split}\]
  For the second part of the error term, we use
  \[\int_{1/Y_6}^{k_6/Y_6} g_1^a(Y_0,t_6;\ee;B) \dd t_6 \ll
  \int_{1/Y_6}^{k_6/Y_6} \frac{1}{Y_0^8} \dd t_6 \ll
  \frac{k_6}{Y_0^8Y_6}\]
  and obtain
  \[\begin{split} &\ll \sum_{\ee} \sum_{k_6|\e_1\e_2\e_3\e_4}
    \frac{|\mu(k_6)|\phid(\e_3)Y_1Y_5Y_6}{k_6\e_2}\cdot\frac{k_6}{Y_0^8Y_6} \\
    &\ll \sum_{\ee} \frac{2^{\omega(\e_1\e_2\e_3\e_4)}\phid(\e_3)
      B^2}{\base 4 4 5 3}\\
    &\ll \sum_{\e_1,\e_2,\e_3} \frac{2^{\omega(\e_1\e_2\e_3)}\phid(\e_3) B\log
      B}{\lambda\base 1 1 1 0}\\
    &\ll \lambda^{-1}B(\log B)^7.
  \end{split}\]
  In total, $E_2(\lambda) \ll \lambda^{-1}B(\log B)^7$.

  Choosing $\lambda=(\log B)^{16/5}$ gives a total error term of $O(B(\log
  B)^{4-1/5})$.
\end{proof}

\section{Estimating $N_{b_1}(B;A)$ -- second step}\label{sec:Nb1}

Let $N_1^b:=N_1^b(\ee, \e_5;B)$ be the main term of $N_0$ in
Lemma~\ref{lem:sum0} summed over $\e_6$ subject to \eqref{eq:cpe6} and
\eqref{eq:bige6}.  We denote the main term of this summed over all $\e_5$ by
$N_2^b:=N_2^b(\ee;B)$.

We remove \eqref{eq:cpe6} by a M\"obius inversion and get
\[N_1^b = \frac{Y_1}{\e_2}
\theta_0(\ee)\frac{\phis(\e_5)}{\phis(\gcd(\e_5,\e_4))}\sum_{k_6|\e_1\e_2\e_3\e_4}
\mu(k_6) A\] where \[A = \sum_{\substack{\e_6' \in \ZZnz\\k_6|\e_6'|> \e_5}}
g_0(Y_0,\e_5/Y_5,k_6\e_6'/Y_6;\ee;B).\] By partial summation, \[A =
\frac{Y_6}{k_6} g_1^b(Y_0, \e_5/Y_5;\ee;B) + O(\sup_{t_6}
g_0(Y_0,\e_5/Y_5,t_6)),\] where the supremum is taken over $t_6$ subject to
$|t_6| > \e_5/Y_6$.

\begin{lemma}\label{lem:sum1b}
  We have
  \[N_1^b = \frac{Y_1Y_6}{\e_2} g_1^b(Y_0, \e_5/Y_5;\ee;B)
    \theta_1^b(\ee)\frac{\phis(\e_5)}{\phis(\gcd(\e_5,\e_4))}+O(R_1^b(\ee,
    \e_5;B))\] with
  \[\theta_1^b(\ee) := \theta_0(\ee)\phis(\e_1\e_2\e_3\e_4) \]
  and \[\sum_{\ee, \e_5} R_1^b(\ee, \e_5;B) \ll B \log B.\]
\end{lemma}

\begin{proof}
  The main term is clear. We apply Lemma~\ref{lem:g}\eqref{it:g0} to deduce
  that the error term can be estimated as \[
  \begin{split}
    \sum_{\ee, \e_5} R_1^b(\ee, \e_5;B) &\ll
    \sum_{\ee, \e_5}\frac{2^{\omega(\e_1\e_2\e_3\e_4)}\phid(\e_3)Y_1Y_6^{1/2}}{\e_2\e_5^{1/2}Y_0}\\
    &= \sum_{\ee, \e_5}\frac{2^{\omega(\e_1\e_2\e_3\e_4)}\phid(\e_3)B^{1/2}}{\e_1^{1/2}\e_2^{1/2}\e_5^{1/2}}\\
    &\ll \sum_{\e_1,\e_2,\e_3,\e_5}
    \frac{2^{\omega(\e_1\e_2\e_3)}\phid(\e_3)B\log
      B}{\base{5/4}{5/4}{5/4}0\e_5^{3/2}}\\
    &\ll B\log B.
  \end{split}\]
  In the last step, we have used \[\e_4 \le \left(\frac{B}{\base 2 2 3
      0\e_5}\right)^{1/4} \cdot \left(\frac{B}{\base 1 1 2 0
      \e_5^3}\right)^{1/4} = \frac{B^{1/2}}{\base
        {3/4}{3/4}{5/4}0\e_5},\] which we obtain using
      \eqref{eq:further_height} and \eqref{eq:bige6}.
\end{proof}

Next, we sum the main term of Lemma~\ref{lem:sum1a} over all suitable
$\e_5$. Apply \cite[Lemma~2]{arXiv:0710.1560} with $\al=0,\,q=1$ to obtain
\[
\begin{split}
  N_2^b ={}&\frac{Y_1Y_6}{\e_2}\theta_1^b(\ee)\sum_{\e_5 \ge 1}
  f_{\e_4,\e_1\e_2\e_3}(\e_5)g_1^b(Y_0,\e_5/Y_5;\ee;B)\\
  ={}&\frac{Y_1Y_5Y_6}{\e_2} g_2^b(Y_0;\ee;B)
  \theta_1^b(\ee)\frac{\phis(\e_1\e_2\e_3)}{\zeta(2)}\prod_{p|\e_1\e_2\e_3\e_4}\left(1-\frac 1 {p^2}\right)^{-1}\\
  &+O\left(\frac{Y_1Y_6}{\e_2}|\theta_1^b(\ee)|(\log
    B)2^{\omega(\e_1\e_2\e_3)} \sup_{t_5}
    g_1^b(Y_0,t_5;\ee;B)\right)\\
  &+O\left(\frac{Y_1Y_5Y_6}{\e_2}|\theta_1^b(\ee)| \int_{0 \le t_5 \le
      1/Y_5} g_1^b(Y_0,t_5;\ee;B)\dd t_5\right),
\end{split}\]
where the supremum is taken over $t_5> 1/Y_5$.

\begin{lemma}\label{lem:sum2b}
  We have \[N_2^b =
  \frac{Y_1Y_5Y_6}{\e_2}g_2^b(Y_0;\ee;B)\theta_2^b(\ee) +
  O(R_2^b(\ee;B))\]
  with
  \[\theta_2^b(\ee):=\theta_1^b(\ee)\frac{\phis(\e_1\e_2\e_3)}{\zeta(2)}\prod_{p|\e_1\e_2\e_3\e_4}\left(1-\frac 1 {p^2}\right)^{-1}\]
  and
  \[\sum_{\ee} R_2^b(\ee;B) \ll B(\log B)^{7-A/4},\]
  where the sum is taken over $\ee$ satisfying \eqref{eq:logBA}.
\end{lemma}

\begin{proof}
  The main term is clear from the discussion before the lemma. The first
  part of the error term makes the contribution \[\ll \sum_{\ee}
  \frac{2^{\omega(\e_1\e_2\e_3)}\phid(\e_3)Y_1Y_6 \log B}{\e_2} \sup
  g_1^b(Y_0,t_5;\ee;B).\]
  We use Lemma~\ref{lem:g}\eqref{it:g1b} and \eqref{eq:logBA} to obtain
  \[\begin{split}
    &\ll \sum_{\ee}
    \frac{2^{\omega(\e_1\e_2\e_3)}\phid(\e_3)Y_1Y_5^{3/4}Y_6
      \log B}{\e_2Y_0}\\
    &= \sum_{\ee} \frac{2^{\omega(\e_1\e_2\e_3)}\phid(\e_3)B^{3/4}
      \log B}{\base {1/2}{1/2}{1/4}{1/2}}\\
    &\ll \sum_{\e_1,\e_2,\e_3} \frac{2^{\omega(\e_1\e_2\e_3)}\phid(\e_3)B(\log
      B)^{1-A/4}}{\base 1 1 1 0}\\
    &\ll B(\log B)^{7-A/4}.
  \end{split}\]

  The contribution of the second term is, using
  Lemma~\ref{lem:g}\eqref{it:g1b} and \eqref{eq:logBA} again,
  \[
  \begin{split}
    &\ll \sum_{\ee} \frac{\phid(\e_3) Y_1Y_5Y_6}{\e_2}
    \int_0^{1/Y_5}\frac{1}{Y_0t_5^{3/4}} \dd t_5\\
      &\ll \sum_{\ee} \frac{\phid(\e_3) Y_1Y_6^{5/4}Y_6}{\e_2Y_0}\\
      &= \sum_{\ee}\frac{\phid(\e_3)B^{3/4}}{\base {1/2}{1/2}{1/4}{1/2}}\\
    &\ll \sum_{\e_1,\e_2,\e_3} \frac{\phid(\e_3)B(\log
      B)^{-A/4}}{\base 1 1 1 0}\\
    &\ll B(\log B)^{3-A/4}.
  \end{split}\]
  This completes the proof of the lemma.
\end{proof}

\section{The final step}\label{sec:final}

By the discussion at the end of Section~\ref{sec:torsor}, we have, for any
$A>0$,
\[N_{U,H}(B) = N_a(B)+N_{b_1}(B;A)+N_{b_2}(B;A).\]

By Lemma~\ref{lem:Nb2}, \[N_{U,H}(B) = N_a(B)+N_{b_1}(B;A)+O_A(B(\log
B)^3(\log \log B)^2).\] Using Lemmas~\ref{lem:sum0}, \ref{lem:sum1a} and
\ref{lem:sum2a} and combining their error
terms shows that
\[N_a(B) = \sum_{\ee \in \EE(B)}
\frac{Y_1Y_5Y_6}{\e_2}g_2^a(Y_0;\ee;B)\theta_2^a(\ee)+O(B(\log B)^{4-1/5}),\]
where \[\EE(t) := \{\ee \in \ZZp^4 \mid \eqref{eq:cpe}, \base 2 2 3 2 \le
t\}\] for any $t\ge 1$, while Lemmas~\ref{lem:sum0}, \ref{lem:sum1b} and
\ref{lem:sum2b} give, choosing $A:=28$,
\[N_{b_1}(B;28) = \sum_{\ee \in \EE(B/(\log B)^{28})}
\frac{Y_1Y_5Y_6}{\e_2}g_2^b(Y_0;\ee;B)\theta_2^b(\ee)+O(B(\log B)^2).\]

Recall the definition~\eqref{eq:g2} of $g_2$.

\begin{lemma}\label{lem:final1}
  We have \[N_{U,H}(B) = \sum_{\ee \in \EE(B)} \frac{Y_1Y_5Y_6}{\e_2}
  g_2(Y_0;\ee;B) \theta(\ee) + O(B(\log B)^{4-1/5})\] where
  \[
  \begin{split}
    \theta(\ee):={}&\frac{\phis(\e_3\e_4)\phis(\e_1\e_2\e_3)\phis(\e_1\e_2\e_3\e_4)}{\zeta(2)}\prod_{p|\e_1\e_2\e_3\e_4}\left(1-\frac
      1
      {p^2}\right)^{-1}\\&\times\left(\sum_{\substack{k_{23}|\e_3\\\cp{k_{23}}{\e_1\e_4}}}\frac{\mu(k_{23})}{k_{23}\phis(\gcd(\e_3,k_{23}\e_2))}\right)
  \end{split}
  \] if the coprimality conditions \eqref{eq:cpe} hold, and $\theta(\ee):=0$
  otherwise.
\end{lemma}

\begin{proof}
  We easily check that $\theta(\ee)$ agrees with $\theta_2^a(\ee)$ and
  $\theta_2^b(\ee)$ for $\ee$ satisfying \eqref{eq:cpe}.

  In view of the discussion before the lemma, it remains to show that
  \[\sum_{\ee \in \EE(B) \setminus \EE(B/(\log B)^{28})}
  \frac{Y_1Y_5Y_6}{\e_2} g_2^b(Y_0;\ee;B)\theta_2^b(\ee)\] makes a negligible
  contribution.

  Indeed, we estimate this as \[\begin{split} &\ll \sum_{\ee \in \EE(B)
      \setminus \EE(B/(\log
      B)^{28})} \frac{\phid(\e_3)Y_1Y_5Y_6}{\e_2Y_0^2} \\
    &=\sum_{\ee \in \EE(B) \setminus \EE(B/(\log B)^{28})}
    \frac{\phid(\e_3)B}{\base 1 1 1 1}\\
    &\ll B(\log B)^3(\log \log B)
  \end{split}\] since we have \[g_2^b(t_0;\ee;B) \ll \int_{0}^{1/t_0^4}
  g_1^b(t_0,t_5;\ee;B) \dd t_5 \ll \int_{0}^{1/t_0^4} \frac{1}{t_0t_5^{3/4}}
  \dd t_5 \ll \frac{1}{t_0^2},\] using Lemma~\ref{lem:g}\eqref{it:g1b} and the
  fact that $g_1^b(t_0,t_5;\ee;B) = 0$ unless $t_5 \ll 1/t_0^4$
  by~\eqref{eq:h}.
\end{proof}

Define \[\EE^*(B) := \{ \ee \in \ZZp^4 \mid \base 2 2 3 2 \le
B, \base 3 3 4 2 > B\}.\]

\begin{lemma}\label{lem:final2}
  We have \[N_{U,H}(B) = \omega_\infty B\sum_{\ee \in \EE^*(B)}
  \frac{\theta(\ee)}{\base 1 1 1 1} + O(B(\log B)^{4-1/5}).\]
\end{lemma}

\begin{proof}
  By Lemma~\ref{lem:g}\eqref{it:g1a}, we have \[g_2(Y_0;\ee;B) \ll
  \int_{|Y_6t_6| > 1} \frac{1}{Y_0^{1/2}|t_6|^{5/4}} \dd t_6 \ll
  \frac{Y_6^{1/4}}{Y_0^{1/2}}.\] Therefore, \[\begin{split}
    \sum_{\ee \in \EE(B)\setminus \EE^*(B)} \frac{Y_1Y_5Y_6}{\e_2}
    g_2(Y_0;\ee;B)\theta(\ee)
    &\ll\sum_{\base 3 3 4 2 \le B}
    \frac{\phid(\e_3)Y_1Y_5Y_6^{5/4}}{\e_2Y_0^{1/2}}\\
    &\ll\sum_{\base 3 3 4 2 \le B} \frac{\phid(\e_3)B^{3/4}}{\base
      {1/4}{1/4}{0}{1/2}}\\
    &\ll\sum_{\e_1, \e_2, \e_3} \frac{\phid(\e_3)B}{\base 1 1 1 0}\\
    &\ll B(\log B)^3.
  \end{split}\] This proves that \[N_{U,H}(B) = \sum_{\ee \in \EE^*(B)}
  \frac{Y_1Y_5Y_6}{\e_2}g_2(Y_0;\ee;B)\theta(\ee) + O(B(\log B)^{4-1/5}).\]

  Comparing the definitions \eqref{eq:g2} and \eqref{eq:G} of the functions
  $g_2$ and $G$ together with the estimation
  \[\begin{split}
    & \sum_{\ee \in \EE^*(B)} \frac{Y_1Y_5Y_6}{\e_2} \int_{h(Y_0,t_1,t_5,t_6)
      \le 1, |Y_6 t_6| \le 1, t_5>0} \dd t_1 \dd t_5 \dd t_6\\
    \ll{} & \sum_{\ee \in \EE^*(B)} \frac{\phid(\e_3)Y_1Y_5Y_6}{\e_2}
    \int_{|Y_6t_6|\le 1}
    \frac{1}{Y_0^8} \dd t_6 \\
    \ll{} & \sum_{\ee \in \EE^*(B)} \frac{\phid(\e_3)Y_1Y_5}{\e_2Y_0^8}\\
    = & \sum_{\ee \in \EE^*(B)} \frac{\phid(\e_3)B^2}{\base 4 4 5 3}\\
    \ll{} & \sum_{\e_1, \e_2, \e_3} \frac{\phid(\e_3)B}{\base 1 1 1 0}\\
    \ll{} & B(\log B)^3
  \end{split}\] shows that \[N_{U,H}(B) = \sum_{\ee \in \EE^*(B)}
  \frac{Y_1Y_5Y_6}{\e_2} G_2(Y_0) \theta(\ee) + O(B(\log B)^{4-1/5}).\]
  Finally, we note that \[\frac{Y_1Y_5Y_6}{\e_2}G_2(Y_0) =
  \frac{Y_1Y_5Y_6}{\e_2Y_0^2}\omega_\infty = \frac{B}{\base 1 1 1
    1}\omega_\infty\] using Lemma~\ref{lem:G}.
\end{proof}

For $\kk = (k_1, k_2, k_3, k_4) \in \ZZp^4$, let \[\Delta_\kk(n) := \sum_{\ee
  \in \ZZp^4, \base{k_1}{k_2}{k_3}{k_4}=n} \frac{\theta(\ee)}{\base 1 1 1
  1}.\]
Consider the Dirichlet series \[F_\kk(s) = \sum_{n=1}^\infty
\frac{\Delta_\kk(n)}{n^{s}} = \sum_{\ee\in \ZZp^4}
\frac{\theta(\ee)}{\base{k_1s+1}{k_2s+1}{k_3s+1}{k_4s+1}}.
\] 
It is absolutely convergent for $\RE(s)>0$. We write it as an Euler product
$F_\kk(s) = \prod_p F_{\kk,p}(s)$, where we compute that $F_{\kk,p}(s)$ is
\begin{multline*}
  (1-1/p)\cdot \left((1+1/p)+\frac{1-1/p}{p^{k_1s+1}-1} +
    \frac{1-1/p}{p^{k_2s+1}-1} + \frac{1-1/p}{p^{k_4s+1}-1}\right.\\
  \left.+\frac{1-1/p}{p^{k_3s+1}-1}\left((1-2/p)+\frac{1-1/p}{p^{k_1s+1}-1} +
      \frac{1-1/p}{p^{k_2s+1}-1} + \frac{1-1/p}{p^{k_4s+1}-1}\right)\right).
\end{multline*} For $\ep > 0$ and $k \in \{(2,2,3,2), (3,3,4,2)\}$ and all $s
\in \CC$ lying in the half-plane $\RE(s) \ge -1/8+\ep$, we have
\[F_{\kk,p}(s) \prod_{j=1}^4\left(1-\frac{1}{p^{k_js+1}}\right) =
1+O_\ep(p^{-1-\ep}).\] We define \[E_\kk(s) := \prod_{j=1}^4 \zeta(k_js+1),
\quad G_\kk(s) := \frac{F_\kk(s)}{E_\kk(s)}\] and note that $F_{\kk}(s)$ has a
meromorphic continuation to $\RE(s) \ge -1/8+\ep$ with a pole of order 4 at
$s=0$.

As in \cite[Lemma~15]{arXiv:0710.1560}, we use a Tauberian theorem to show that
\[M_\kk(t):= \sum_{n \le t} \Delta_\kk(n)\] can be estimated as
\[\frac{G_\kk(0) P(\log t)}{4! \prod_{j=1}^4 k_j} + O(t^{-\delta})\] for some
$\delta>0$ and $P$ a monic polynomial of degree 4.

Using Lemma~\ref{lem:final2} and the definitions of $\Delta_\kk$ and $M_\kk$,
\[\begin{split}
  N_{U,H}(B) &= \omega_\infty B \sum_{n \le B}
  (\Delta_{(2,2,3,2)}(n)-\Delta_{3,3,4,2}(n)) + O(B(\log B)^{4-1/5})\\
  &= \omega_\infty G_\kk(0) \frac{1}{4!}\left(\frac{1}{2^3\cdot 3}-\frac{1}{2
      \cdot 3^2\cdot 4}\right) B(\log B)^4 + O(B(\log B)^{4-1/5})
\end{split}\] 
Since \[\alpha(\tS) = \frac{1}{864} =
\frac{1}{4!}\left(\frac{1}{2^3\cdot 3}-\frac{1}{2 \cdot 3^2\cdot
    4}\right)\] and \[G_\kk(0)= \prod_p\left(1-\frac 1
  p\right)^5\left(1+\frac 5 p +\frac 1{p^2}\right),\] this completes the proof
of the theorem.

\bibliographystyle{amsalpha}

\bibliography{manin_dp5_a2}

\end{document}